\documentclass[a4, 12pt]{amsart}
\usepackage[mathscr]{eucal}
\usepackage{amssymb}
\usepackage{latexsym}
\usepackage{amsthm}
\theoremstyle{plain}

\usepackage{graphics, color}

\usepackage{epsfig}

\usepackage{graphicx}
\usepackage{color}

\usepackage{graphicx,color}
\usepackage{amsmath, amssymb, graphics}

\usepackage{graphicx}

\theoremstyle{plain}
\newtheorem{theorem}{Theorem}[section]

\newtheorem{remark}{Remark}[section]
\newtheorem{lemma}{Lemma}[section]
\newtheorem{example}{Example}[section]
\newtheorem{definition}{Definition}[section]

\makeatletter
\@addtoreset{equation}{section}

\title[Compact embedded $\lambda$-torus]
{Compact embedded $\lambda$-torus in Euclidean spaces}
\author{Qing-Ming Cheng and  Guoxin Wei}
\address{Qing-Ming Cheng \\ Department of Applied Mathematics, Faculty of Sciences,
Fukuoka  University, 814-0180, Fukuoka,  Japan, cheng@fukuoka-u.ac.jp}
\address{Guoxin Wei \\  School of Mathematical Sciences, South China Normal University,
510631, Guangzhou,  China, weiguoxin@tsinghua.org.cn}

\begin{document}
\maketitle

\begin{abstract}
\noindent In this paper, we construct compact embedded $\lambda$-hypersurfaces
with the topology of torus which are called $\lambda$-torus in Euclidean spaces $\mathbb {R}^{n+1}$.
\end{abstract}



\footnotetext{The first author was partially  supported by JSPS Grant-in-Aid for Scientific Research (B): No. 24340013
and Challenging Exploratory Research No. 25610016.
The second author was partly supported by grant No. 11371150 of NSFC.}

\section {Introduction}

\noindent
Let $X: M\rightarrow \mathbb{R}^{n+1}$ be an  $n$-dimensional smooth immersed hypersurface in the $(n+1)$-dimensional
Euclidean space $\mathbb{R}^{n+1}$. In \cite{cw2}, Cheng and Wei have introduced the
weighted volume-preserving mean curvature flow, which is defined as the following:
a family $X(t)$ of smooth immersions:
$$
X(t)=X(\cdot, t):M\to  \mathbb{R}^{n+1}
$$
with  $X(0)=X(\cdot, 0)=X(\cdot)$ is called {\it a weighted volume-preserving mean curvature flow} if they satisfy
\begin{equation}
\dfrac{\partial X(t)}{\partial t}=-\alpha(t) N(t) +\mathbf{H}(t),
\end{equation}
 where
$$
\alpha(t) =\dfrac{\int_MH(t)\langle N(t), N\rangle e^{-\frac{|X|^2}2}d\mu}{\int_M\langle N(t), N\rangle e^{-\frac{|X|^2}2}d\mu},
$$
 $\mathbf{H}(t)=\mathbf{H}(\cdot,t)$ and $N(t)$ denote the mean curvature vector  and the  unit normal vector of hypersurface
 $M_t=X(M^n,t)$ at point $X(\cdot, t)$, respectively  and   $N$ is the  unit normal vector of $X:M\to  \mathbb{R}^{n+1}$.\\
One  can prove  that the flow (1.1) preserves the weighted volume $V(t)$ defined by
$$
V(t)=\int_M\langle X(t),N\rangle e^{-\frac{|X|^2}{2}}d\mu.
$$
\noindent
{\it The weighted area functional}
$A:(-\varepsilon,\varepsilon)\rightarrow\mathbb{R}$ is defined by
$$
A(t)=\int_Me^{-\frac{|X(t)|^2}{2}}d\mu_t,
$$
where $d\mu_t$ is the area element of $M$ in the metric induced by $X(t)$.

\noindent
Let $X(t):M\rightarrow \mathbb{R}^{n+1}$ with $X(0)=X$ be a  variation of $X$.
If $V(t)$ is constant for any $t$, we call  $X(t):M\rightarrow \mathbb{R}^{n+1}$ {\it  a weighted volume-preserving variation of $X$}.
Cheng and Wei \cite{cw2} have proved that  $X:M\rightarrow \mathbb{R}^{n+1}$ is a critical point of the weighted area functional $A(t)$
for all weighted volume-preserving variations if and only  if there exists constant $\lambda$ such that
\begin{equation}\label{eq:10-3-1}
\langle X, N\rangle +H=\lambda.
\end{equation}
An immersed hypersurface $X(t):M\rightarrow \mathbb{R}^{n+1}$ is called {\it a $\lambda$-hypersurface of
weighted volume-preserving mean curvature flow} if the equation (1.2) is satisfied.

\begin{remark}
If $\lambda=0$,  $\lambda$-hypersurfaces are   self-shrinkers of mean curvature flow. Hence,  $\lambda$-hypersurfaces are
a generalization of self-shrinkers of the mean curvature flow. For research on self-shrinkers of mean curvature flow, see \cite{CL, CW, CM, DX,  H2}.
\end{remark}
\begin{example}
The $n$-dimensional sphere $S^n(r)$ with radius $r>0$ is a compact $\lambda$-hypersurface in $\mathbb{R}^{n+1}$
with $\lambda=\frac{n}r-r$.
\end{example}
\begin{example}
For $1\leq k\leq n-1$, the $n$-dimensional cylinder  $S^k(r)\times \mathbb{R}^{n-k}$ with radius $r>0$ is a complete
and non-compact $\lambda$-hypersurface in $\mathbb{R}^{n+1}$ with $\lambda=\frac{k}r-r$.
\end{example}
\begin{example}
The $n$-dimensional Euclidean space $\mathbb{R}^{n}$ is a complete and non-compact $\lambda$-hypersurface
in $\mathbb{R}^{n+1}$ with $\lambda=0$.
\end{example}

\vskip 1mm
\noindent
It is well known that many interesting  examples of compact self-shrinker of mean curvature flow are found recently.
For instance,  Angenent  \cite{[A]} has constructed compact embedded self-shrinker
$$
X:S^1\times S^{n-1}\to\mathbb{R}^{n+1}.
$$
Drugan \cite{D} has discovered self-shrinker, which is a topological sphere
$$
X:S^2\to\mathbb{R}^{3}
$$
and complete self-shrinkers with higher genus in $\mathbb{R}^{3}$ are constructed by
Kapouleas, Kleene and M{\o}ller \cite{KKM} (see  \cite{KM} also) and Nguyen \cite{N1}-\cite{N3}.

\noindent
Our purpose in this paper is to construct compact embedded $\lambda$-hypersurfaces in $\mathbb{R}^{n+1}$.

\begin{theorem}
For $n\geq 2$ and any $\lambda>0$, there exists an embedded rotational $\lambda$-hypersurface
\begin{equation}
X: M\rightarrow \mathbb{R}^{n+1},
\end{equation}
which is called $\lambda$-torus.
\end{theorem}

\section {Equations of rotational $\lambda$-hypersurfaces in $\mathbb{R}^{n+1}$}

\noindent
Let $\gamma(s)=(x(s), r(s))$, $s\in (a,b)$ be a curve with $r>0$ in
the upper half plane $\mathbb{H}=\{x+ir| \  r>0, \  x\in \mathbb{R}, \ i=\sqrt{-1}\}$,
where $s$ is arc length parameter of $\gamma(s)$.
We consider  a rotational hypersurface $X: (a,b)\times S^{n-1}(1) \hookrightarrow \mathbb{R}^{n+1}$
in $\mathbb{R}^{n+1}$ defined by
\begin{equation}\label{eq:10-30-1}
X:  (a,b)\times S^{n-1}(1)\hookrightarrow \mathbb{R}^{n+1}, \ \ X(s, \alpha)=(x(s), r(s)\alpha) \in \mathbb{R}^{n+1}
\end{equation}
where  $S^{n-1}(1)$ is the ($n-1$)-dimensional unit sphere (cf. \cite{[DD]}).

\noindent
By a direct calculation, one has the unit normal vector
\begin{equation}\label{eq:10-30-2}
N=(-r^{\prime}, x^{\prime}\alpha)
\end{equation}
and the mean curvature
\begin{equation}\label{eq:10-30-3}
H=-x^{\prime\prime}r^{\prime}+x^{\prime}r^{\prime\prime}-\dfrac{n-1}{r}x^{\prime}.
\end{equation}
Therefore, we know from \eqref{eq:10-30-1} and \eqref{eq:10-30-2}

\begin{equation}\label{eq:10-30-4}
\langle X, N\rangle=-xr^{\prime}+rx^{\prime}.
\end{equation}
Hence,  $X:  (a,b)\times S^{n-1}(1)\hookrightarrow \mathbb{R}^{n+1}$ is a $\lambda$-hypersurface in $\mathbb{R}^{n+1}$,
if and only if,  from \eqref{eq:10-3-1}, \eqref{eq:10-30-3} and \eqref{eq:10-30-4},
\begin{equation}\label{eq:9-16-11}
-x^{\prime\prime}r^{\prime}+x^{\prime}r^{\prime\prime}-\dfrac{n-1}{r}x^{\prime}-xr^{\prime}+rx^{\prime}=\lambda.
\end{equation}
Since $s$ is arc length parameter of the profile curve $\gamma(s)=(x(s), r(s))$, we have
\begin{equation}\label{eq:9-4-11}
(x^{\prime})^2+(r^{\prime})^2=1,
\end{equation}
Thus, it follows that
\begin{equation}\label{eq:9-16-12}
x^{\prime}x^{\prime\prime}+r^{\prime}r^{\prime\prime}=0.
\end{equation}
The signed curvature $\kappa(s)$ of the profile curve $\gamma(s)=(x(s), r(s))$ is given by
\begin{equation}\label{eq:12-4-2}
\kappa(s)=-\dfrac{x^{\prime\prime}}{r^{\prime}},
\end{equation}
and it is known that the integral of the signed curvature $\kappa(s)$ measures the total rotation of the tangent vector of
 $\gamma(s)$.
From \eqref{eq:9-16-11} and \eqref{eq:9-16-12}, one has
\begin{equation}\label{eq:9-4-2}
x^{\prime\prime}=-r^{\prime}\bigl[xr^{\prime}+\bigl(\dfrac{n-1}{r}-r\bigl)x^{\prime}+\lambda\bigl].
\end{equation}
Hence, we have

\begin{equation}\label{eq:9-4-3}
  {\begin{cases}
     (x^{\prime})^2+(r^{\prime})^2=1\\[2mm]

     -\dfrac{x^{\prime\prime}}{r^{\prime}}=xr^{\prime}+\bigl(\dfrac{n-1}{r}-r\bigl)x^{\prime}+\lambda.
     \end{cases}}
  \end{equation}

  \vskip 1mm

 \noindent First of all, we consider several special solutions of \eqref{eq:9-4-3}.

 \vskip 1mm
\begin{enumerate}
\item  $(x,r)=(0, s)$ is a solution. \\
This curve corresponds to the  hyperplane through $(0, 0)$ and $\lambda=0$.\vskip 1mm

\item $(x,r)=(a\cos\frac{s}{a}, a\sin\frac{s}{a})$ is a solution, where $a=\dfrac{\sqrt{\lambda^2+4n}-\lambda}{2}$. \\
This circle $x^2+r^2=a^2$ corresponds to a sphere $S^n(a)$ with radius $a$.

\item  $(x,r)=(-s, a)$ is a solution, where $a=\dfrac{\sqrt{\lambda^2+4(n-1)}-\lambda}{2}$.\\
This straight line corresponds to  a cylinder $S^{n-1}(a)\times \mathbb{R}$.\vskip 1mm
\end{enumerate}

\noindent
Next, we consider to find several other solutions of \eqref{eq:9-4-3} besides the above three special solutions.
In fact, our purposes are to study   properties of the profile curve and
to find a simple and closed profile curve $\gamma$ in the upper half plane
$\mathbb{H}=\{x+ir |\ r>0, \  x \in \mathbb{R},  i=\sqrt{-1}\}$.

\noindent
From now,  we consider general behavior of profile curve $\gamma$.
As long as $x>0$, $x^{\prime}>0$ and $r<\dfrac{\sqrt{\lambda^2+4(n-1)}+\lambda}{2}$, one has from \eqref{eq:12-4-2} and \eqref{eq:9-4-3} that
 \begin{equation*}
 \kappa=-\dfrac{x^{\prime\prime}}{r^{\prime}}=xr^{\prime}+\bigl(\dfrac{n-1}{r}-r\bigl)x^{\prime}+\lambda>0,
 \end{equation*}
 that is, the curve $\gamma$ bends upwards, so that the curve will be convex.

 \noindent
At $s=0$,
the profile curve  $\gamma_\delta(s)=(x_\delta(s), r_\delta(s))$ with
$(x_{\delta}(0), r_{\delta}(0))=(0, \delta)$ and the initial unit tangent vector
$(x_{\delta}^{\prime}(0), r_{\delta}^{\prime}(0))=(1, 0)$
will initially bend upwards, where
$\delta<\dfrac{\sqrt{\lambda^2+4(n-1)}+\lambda}{2}$.
 We give the following  definition of $s_1$.

\begin{definition}\label{definition 2.1}

Let $s_1=s_1(\delta)>0$ be the arc length of the first time, if any, at which either $x_\delta=0$ or the unit tangent vector  is $(1, 0)$, or
$(-1,0)$, that is, either the curve $\gamma_\delta$ hits $r$-axis, or the curve $\gamma_\delta$ has a horizontal tangent. If this never happens, we take $s_1(\delta)=S(\delta)$, where $\gamma_\delta=(x_\delta,r_\delta): [0, S(\delta))\rightarrow \mathbb{R}^2 $ is the maximal solution of \eqref{eq:9-4-3} with initial value $(x_\delta(0),r_\delta(0),x_\delta^{\prime}(0))=(0, \delta, 1)$.

\end{definition}

\noindent
From this definition \ref{definition 2.1}, we have $r^{\prime}(s)>0$ in  $(0, s_1)$.
Hence, this curve can be written as a  graph of $x=f_\delta(r)$,
where $\delta<r<r_\delta(s_1)$. If $f_\delta^{\prime}(r)=\frac{dx}{dr}=0$, i.e.,
the profile curve $\gamma_{\delta}$ has a vertical tangent, then it follows from \eqref{eq:9-4-3} that

\begin{equation}
\aligned
 f_\delta^{\prime\prime}(r)&=\dfrac{d^2f_\delta(r)}{dr^2}
 =-\frac{1}{(r^{\prime})^3}[xr^{\prime}+(\dfrac{n-1}{r}-r)x^{\prime}+\lambda]\\
&=-\dfrac{1}{(r^{\prime})^3}(xr^{\prime}+\lambda)<0.
\endaligned
\end{equation}
This means that $f_\delta(r)$ can only have local maximum. Thus, $f_\delta(r)$ can have at most one critical point, which must be a
maximum  point.

\noindent Next, we shall prove that there exist $\delta>0$
such that $\gamma_\delta([0,s_1(\delta)])$ is a simple curve
in the first quadrant which begins and ends on the $r$-axis,
and whose tangent vectors on the $r$-axis are horizontal.
From \eqref{eq:9-4-3}, one can get that the profile curve $\gamma$
obtained by reflecting $\gamma_\delta([0,s_1(\delta)])$ in the $r$-axis is a simple and closed curve in the upper half plane.

\section {An estimate on upper bounds of $r_\delta(s_1)$}

\noindent
We will consider behavior of profile curve $\gamma$ as $\delta>0$ is small enough in order to estimate supremum of $r_\delta(s_1)$.
Since $\delta$ is very small, we define

\begin{equation}\label{eq:10-19-2}
  {\begin{cases}
      \xi (t)=\dfrac{1}{\delta}x(\delta t)\\[2mm]
      \rho(t)=\dfrac{1}{\delta}(r(\delta t)-\delta).
     \end{cases}}
  \end{equation}
From \eqref{eq:9-4-3}, we have
\begin{equation}\label{eq:9-6-11}
  {\begin{cases}
      (\xi^{\prime})^2&+\ (\rho^{\prime})^2=1\\[2mm]
      \dfrac{\xi^{\prime\prime}}{- \rho^{\prime}}&=\dfrac{\xi^{\prime\prime}}{-\sqrt{1-(\xi^{\prime})^2}}\\
      &=\delta^2\xi\rho^{\prime}+\biggl(\dfrac{n-1}{1+\rho}-\delta^2(1+\rho)\biggl)\xi^{\prime}+\lambda \delta\\
      &=\dfrac{n-1}{1+\rho}\xi^{\prime}+\lambda \delta+O(\delta^2)
     \end{cases}}
  \end{equation}
and
\begin{equation}
\xi (0)=0, \ \ \rho(0)=0, \ \ \xi^{\prime}(0)=1.
\end{equation}

\noindent
We consider  equations
\begin{equation}\label{eq:12-4-3}
  {\begin{cases}
      (\xi^{\prime})^2+(\rho^{\prime})^2=1\\[2mm]
      \dfrac{\xi^{\prime\prime}}{- \rho^{\prime}}=\dfrac{\xi^{\prime\prime}}{-\sqrt{1-(\xi^{\prime})^2}}=\dfrac{n-1}{1+\rho}\xi^{\prime},
     \end{cases}}
  \end{equation}
with
\begin{equation}
\xi (0)=0, \ \ \rho(0)=0, \ \ \xi^{\prime}(0)=1.
\end{equation}
From \eqref{eq:12-4-3}, one gets
$$
1-(\rho^{\prime})^2=\dfrac1{(1+\rho)^{2(n-1)}}.
$$
If $\rho(t)  \rightarrow +\infty$ as $t \rightarrow+\infty$, we have
$$
\rho^{\prime}=\dfrac{d\rho}{dt}=\dfrac{dr}{ds} \rightarrow 1.
$$
If $\rho(t)$ is bounded, we have
\begin{equation}
c_1\leq \dfrac{\xi^{\prime\prime}}{-\xi^{\prime}\sqrt{1-(\xi^{\prime})^2}}=\dfrac{n-1}{1+\rho}\leq(n-1)
\end{equation}
where $c_1>0$ is a  constant. By a direct calculation, we   get
\begin{equation}\label{eq:12-4-4}
\tanh (c_1 t)\leq \sqrt{1-(\xi^{\prime})^2}=\rho^{\prime}.
\end{equation}
Hence, ${\rho^{\prime}}  \rightarrow 1$ as  $t\rightarrow +\infty$.

\noindent
Since the solution of \eqref{eq:9-6-11} depends smoothly on the parameter $\delta$, we may obtain from \eqref{eq:12-4-4} that
 there is a $T>0$ such that for  all sufficiently  small $\delta>0$ , one has $T \delta<S(\delta)$ and at $s=T\delta$,

\begin{equation}\label{eq:10-6-1}
  r_{\delta}^{\prime}(T\delta)\geq
  {\begin{cases}
      \dfrac{\sqrt{3}}{2},\ \ \ \ \ \ \ \ \ \ \ \ \ \ \ \ \ \ \ \ \ \ \ \  \ {\mbox{if}} \ \lambda>\dfrac{\pi}{3\sqrt{n-1}}    \\[2mm]
      \sin(\dfrac{\pi}{2}-\lambda\dfrac{\sqrt{n-1}}{2}),\ \ \ \ \  \ {\mbox{if}} \ \lambda\leq \dfrac{\pi}{3\sqrt{n-1}} .
     \end{cases}}
  \end{equation}
and $x_\delta=O(\delta)$, $r_\delta=\delta+O(\delta)$ from \eqref{eq:10-19-2}.

\begin{lemma}\label{lemma 2}
For $0<s<s_1$, we have
\begin{equation}
x_\delta(s)\leq C_1 \delta,\  r_\delta(s)\leq \sqrt{n-1}+\frac{\pi}{2\lambda}  \   {\mbox{and}}  \ s_1(\delta)<\infty
\end{equation}
if $\delta$ is small enough, where  constant $C_1$ does not depend on $\delta$.
\end{lemma}

\begin{proof}
For $0<s<s_1=s_1(\delta)$, we have $r^{\prime}>0$, so that $\gamma_\delta$ can
be written as  a graph of $x=f_{\delta}(r)$.
From \eqref{eq:9-4-3}, we have
\begin{equation}\label{eq:10-5-1}
\aligned
-\dfrac{x^{\prime\prime}}{r^{\prime}}&=xr^{\prime}+(\dfrac{n-1}{r}-r)x^{\prime}+\lambda\\
                                        &\geq(\dfrac{n-1}{r}-r)x^{\prime}
\endaligned
\end{equation}
because of $x>0$, $\lambda>0$ and $r^{\prime}>0$.

\noindent
Letting $r_T=r_\delta(T\delta)$ and for $r$ satisfying $x^{\prime}(\widetilde{r})>0$ if  $r_T<\widetilde{r}<r$,
integrating  \eqref{eq:10-5-1} from $r_T$  to  $r>r_T$,
we have
\begin{equation}
\dfrac{x^{\prime}(s(r_T))}{x^{\prime}(s(r))}\geq e^{\frac{r_T^2-r^2}{2}}(\dfrac{r}{r_T})^{n-1},
\end{equation}
i.e., for all $r$ such that $x^{\prime}>0$ on $(r_T, r)$, we have
\begin{equation}\label{eq:9-17-11}
x^{\prime}(s(r))\leq (\dfrac{r_T}{r})^{n-1}e^{\frac{r^2-r_T^2}{2}}x^{\prime}(s(r_T)).
\end{equation}
We will prove that the above $r$ satisfies

Claim: $r\leq\sqrt{n-1}$ if $\delta$ is small enough.

\noindent
If   $r(s_1)\leq\sqrt{n-1}$, then we know that
  this result is obvious because of  $r^{\prime}(s)>0$.

\noindent
We only need to consider the case of  $r(s_1)>\sqrt{n-1}$.
In this case, there exists an  $s_3$ such that $0<s_3<s_1$ and $r(s_3)=\sqrt{n-1}$.

\noindent
If $x^{\prime}(s_3)\leq0$, we have  $x^{\prime}(s)<0$ for $s_3<s<s_1$.
Hence, $r<r(s_3)=\sqrt{n-1}$ holds because of $x^{\prime}(r)>0$.

\noindent
If $x^{\prime}(s_3)>0$,
  for $s\in (T\delta, s_3)$, $r^{\prime}(s)>0$ holds.  One has
$0<r(s)<\sqrt{n-1}$ and
\begin{equation}\label{eq:12-4-5}
 \kappa=-\dfrac{x^{\prime\prime}}{r^{\prime}}=xr^{\prime}+\bigl(\dfrac{n-1}{r}-r\bigl)x^{\prime}+\lambda\geq\lambda
 \end{equation}
because of  $r(s_3)=\sqrt{n-1}$.
By integrating  \eqref{eq:12-4-5} from $s=T\delta$ to $s_3$, we obtain
\begin{equation}
\int^{s_3}_{T\delta}\kappa ds\geq\lambda (s_3-T\delta)\geq\lambda (r(s_3)-r(T\delta))\geq\lambda(\sqrt{n-1}-r_T).
\end{equation}
Here we use that the length of the profile curve $\gamma(s)$
 from the point $\gamma(T\delta)$ to the point $\gamma(s_3)$  is not less than the Euclidean distance between these two points.
Therefore, we have
\begin{equation}\label{eq:12-5-1}
r_T\geq\sqrt{n-1}-\dfrac{1}{\lambda}\int^{s_3}_{T\delta}\kappa ds.
\end{equation}
If $\lambda>\frac{\pi}{3\sqrt{n-1}}$, we have
$$
\dfrac{\pi}{2}-\lambda\dfrac{\sqrt{n-1}}{2}<\dfrac{\pi}3.
$$
If $\lambda\leq\frac{\pi}{3\sqrt{n-1}}$, then we have
$$
\dfrac{\pi}{2}-\lambda\dfrac{\sqrt{n-1}}{2}\geq \dfrac{\pi}3.
$$
Since the integral
$$
\int^{s_3}_{T\delta}\kappa ds
$$
measures the total rotation of the tangent vector of $\gamma_{\delta}$ from $T\delta$ to $s_3$, from \eqref{eq:10-6-1}, we know

$$
\int^{s_3}_{T\delta}\kappa ds\leq \dfrac{\pi}2-\max\{\dfrac{\pi}3, (\dfrac{\pi}2-\lambda\dfrac{\sqrt{n-1}}2\}.
$$
Thus,  for $\lambda>\frac{\pi}{3\sqrt{n-1}}$, we obtain  from \eqref{eq:12-5-1}
\begin{equation}
r_T\geq\sqrt{n-1}-\dfrac{1}{\lambda}\int^{s_3}_{T\delta}\kappa ds\geq\sqrt{n-1}-\dfrac{1}{\lambda}(\frac{\pi}{2}-\frac{\pi}{3})
\geq\frac{\sqrt{n-1}}{2}.
\end{equation}
It is impossible because  $r_T=r(T\delta)=O(\delta)$ is very small.

\noindent
For  $\lambda\leq\frac{\pi}{3\sqrt{n-1}}$,  we have from \eqref{eq:12-5-1}
\begin{equation}
r_T\geq\sqrt{n-1}-\dfrac{1}{\lambda}\int^{s_3}_{T\delta}\kappa ds\geq\sqrt{n-1}-\dfrac{1}{\lambda}(\frac{\pi}{2}-
(\dfrac{\pi}{2}-\lambda\dfrac{\sqrt{n-1}}{2}))
=\frac{\sqrt{n-1}}{2}.
\end{equation}
It is also  impossible because of  $r_T=r(T\delta)=O(\delta)$.

\noindent
From the above arguments, we complete the proof of the claim.

\noindent
By a direct calculation, we have from \eqref{eq:9-17-11} and  \eqref{eq:10-6-1}

\begin{equation}
x^{\prime}(s(r))\leq x^{\prime}(s(\delta_T))\leq \dfrac{1}{2}
\end{equation}
for $r\in (r_T, \sqrt{n-1})$ with $x^{\prime}>0$ on $(r_T, r)$
since
$$
(\dfrac{r_T}{r})^{n-1}e^{\frac{r^2-r_T^2}{2}}
$$
is a decreasing function of $r$ in $(r_T, \sqrt{n-1})$.
Hence, for $r\in (r_T, \sqrt{n-1})$ one has
\begin{equation}\label{eq:10-6-2}
\dfrac{x^{\prime}}{r^{\prime}}\leq2x^{\prime}.
\end{equation}
Since $\dfrac{dx}{dr}=f_\delta^{\prime}(r)=\dfrac{x^{\prime}}{r^{\prime}}$,
we have from \eqref{eq:9-17-11} and \eqref{eq:10-6-2} that

\begin{equation}
\int_{r_T}^r f_\delta^{\prime}(r) dr=\int_{r_T}^r \dfrac{x^{\prime}}{r^{\prime}} dr\leq 2\int_{r_T}^r x^{\prime}dr
\leq 2\int_{r_T}^r(\dfrac{r_T}{r})^{n-1}e^{\frac{r^2-r_T^2}{2}}x^{\prime}(s(r_T))dr.
\end{equation}
Hence, we obtain
\begin{equation}\label{eq:9-17-12}
\begin{aligned}
f_\delta(r)&\leq f_\delta(r_T)+2\int_{r_T}^r(\dfrac{r_T}{r})^{n-1}e^{\frac{r^2-r_T^2}{2}}dr\\
&\leq f_\delta(r_T)+c_2\int_{r_T}^r(\dfrac{r_T}{r})^{n-1}dr\leq c_3r_T\leq C_1\delta
\end{aligned}
\end{equation}
if $\delta>0$ is small enough,  where $c_2$, $c_3$ and $C_1$ are constants.

\noindent
Since $x_\delta(s)$ gets its maximum at $x^{\prime}(s)=0$, we conclude from \eqref{eq:9-17-12} that
\begin{equation}
0\leq x_\delta(s)\leq C_1\delta, \ \ \ \ \ \ \  {\mbox{for}}\ 0<s<s_1.
\end{equation}
If  $x^{\prime}(s_3)=0$ at $s=s_3$, then  we have $x^{\prime}(s)<0$ for  $s_3<s<s_1$.
According to the argument in
the above claim, we know   $r(s_3)\leq\sqrt{n-1}$ as long as $\delta>0$ is small enough.
If there exists $s_3<s_2<s_1$ such that $r(s_2)=\sqrt{n-1}$, then $r(s)>\sqrt{n-1}$
and $x^{\prime}(s)<0$ for $s_2<s<s_1$.
Hence, we have
 $$
 (\frac{n-1}{r(s)}-r(s))x^{\prime}(s)>0
 $$
 and
\begin{equation}\label{eq:12-6-1}
\kappa=-\dfrac{x^{\prime\prime}}{r^{\prime}}=xr^{\prime}+\bigl(\dfrac{n-1}{r}-r\bigl)x^{\prime}+\lambda>\lambda,
\end{equation}
 for $s_2<s<s_1$.
By integrating \eqref{eq:12-6-1} from $s=s_2$ to $s_1$, we have
\begin{equation}
\frac{\pi}{2}\geq \int^{s_1}_{s_2}\kappa ds\geq \lambda (s_1-s_2)\geq \lambda (r(s_1)-r(s_2))
= \lambda (r(s_1)-\sqrt{n-1}),
\end{equation}
that is,
\begin{equation}
r_\delta(s_1)\leq\sqrt{n-1}+\frac{\pi}{2\lambda}.
\end{equation}
Hence, we get  that $x_\delta(s)$ and $ r_\delta(s)$ are bounded.  We know $s_1(\delta)<\infty$.
This finishes  the proof of the lemma 3.1.
\end{proof}

\begin{lemma}\label{lemma 3.2}
For $\delta>0$ small enough, one has $x_\delta(s_1)=0$.
\end{lemma}

\begin{proof}
If not, there exists a sequence $\delta_m\rightarrow 0^{+}$ for which $x_{\delta_m}(s)>0$.
Putting  $f_m(r)=f_{\delta_m}(r)$,
the function $x_{\delta_m}=f_m(r)$ is  defined for $\delta_m<r<\sqrt{n-1}+\frac{\pi}{2\lambda}$
and  satisfies
\begin{equation}\label{eq:9-10-11}
\dfrac{f_m^{\prime\prime}(r)}{1+(f_m^{\prime}(r))^2}+(\dfrac{n-1}{r}-r)f_m^{\prime}(r)
+f_m(r)+\dfrac{\lambda}{\sqrt{1+(f_m^{\prime}(r))^2}}=0.
\end{equation}
From the lemma \ref{lemma 2}, we get that $f_m(r)$ satisfies
$0<f_m(r)=x_{\delta_m}(r)\leq C_1 \delta_m\rightarrow0$.
Thus,  we know that $\gamma_{\delta_m}=(x_{\delta_m}, r_{\delta_m})$ gets
close to the $r$-axis, its tangents also must converge to the $r$-axis.
Hence, we have $x_{\delta_m}(r)=f_m(r)$ and $f_m^{\prime}(r)$ converge to zero on compact intervals.

\noindent
On the other hand, from \eqref{eq:9-10-11}, we have
\begin{equation}
f_m^{\prime\prime}(r)\rightarrow-\lambda<0.
\end{equation}
Therefore,  for $\delta_m>0$ small enough,
\begin{equation}
f_m^{\prime\prime}(r)<-\dfrac{\lambda}{2}<0,
\end{equation}
$f_m^{\prime}(r)$ is a monotone decreasing function. This is impossible. Hence, there exists $\delta>0$ small enough  such that $x_\delta(s_1)=0$.
\end{proof}

\noindent
According to the lemma \ref{lemma 2} and the lemma \ref{lemma 3.2}, we know that there exists  $\delta_0>0$ small enough  such that  $x_{\delta_0}(s_1)=0$ and $s_1(\delta_0)<\infty$.
Since  solutions of  \eqref{eq:9-4-3} depend smoothly on the initial value, we define $\delta^{*}$ as the following

\begin{definition}\label{definition 11-3-1}
\begin{equation}
\delta^{*}=\sup\{\delta>0:   x_{\delta}(s_1)=0 \ {\mbox{and}} \ s_1=s_1({\delta})<\infty \}.
\end{equation}

\end{definition}

\begin{lemma}
\begin{equation}
\sup\limits_{\delta_0<\delta<\delta^{*}} r_\delta(s_1)<+\infty.
\end{equation}
\end{lemma}
\begin{proof}
From the lemma \ref{lemma 3.2}, we know $r_\delta(s_1)$ is bounded if $\delta>0$ is small enough.
If $r_{\delta_m}(s_1)\rightarrow +\infty$ for some sequence $\delta_m\rightarrow \delta^{*}$, we will
prove that it is impossible. In fact, according to the mean value theorem,
there exists an $s_3$ such that $0<s_3<s_1$ and $x_{\delta_m}^{\prime}(s_3)=0$
because of  $x_{\delta_m}(s_1)=0$.
We should remark that $s_1$ and $s_3$ depend on $\delta_m$.
Furthermore,  $x_{\delta_m}^{\prime}(s)<0$ for $s_3<s<s_1$
and we have
$$
\dfrac{n-1}{r_{\delta_m}}-r_{\delta_m}<0,
$$
for $s_3<s<s_1$.
Otherwise, there exists an $s_4$ such that $r_{\delta_m}(s_4)=\sqrt{n-1}$ with $s_4>s_3$.
Thus,  $x_{\delta_m}^{\prime}(s)<0$ and $\frac{n-1}{r_{\delta_m}}-r_{\delta_m}<0$
for $s_4<s<s_1$. From the equations \eqref{eq:9-4-3}, we have
\begin{equation}\label{eq:10-31-1}
\kappa=-\dfrac{x^{\prime\prime}_{\delta_m}}{r^{\prime}_{\delta_m}}
=x_{\delta_m}r^{\prime}_{\delta_m}
+\bigl(\dfrac{n-1}{r_{\delta_m}}-r_{\delta_m}\bigl)x^{\prime}_{\delta_m}+\lambda>\lambda.
\end{equation}
By integrating \eqref{eq:10-31-1} from $s_4$ to $s_1$, we get
\begin{equation}
\begin{aligned}
\dfrac{\pi}{2}&\geq\int_{s_4}^{s_1}\kappa ds\geq \lambda\bigl(s_1-s_4\bigl)\\
&\geq\lambda\bigl(r_{\delta_m}(s_1)-r_{\delta_m}(s_4)\bigl)\\
&=\lambda\bigl(r_{\delta_m}(s_1)-\sqrt{n-1}\bigl).
\end{aligned}
\end{equation}
Hence, we have
\begin{equation}
r_{\delta_m}(s_1)\leq \frac{\pi}{2\lambda}+\sqrt{n-1}.
\end{equation}
This is impossible since $r_{\delta_m}(s_1)\rightarrow +\infty$.
Hence,
$$
\frac{n-1}{r_{\delta_m}}-r_{\delta_m}<0
$$ for $s_3<s<s_1$.
For $s_3<s<s_1$, we have
\begin{equation}\label{eq:12-5-2}
\kappa=-\dfrac{x^{\prime\prime}_{\delta_m}}{r^{\prime}_{\delta_m}}
=x_{\delta_m}r^{\prime}_{\delta_m}+\bigl(\dfrac{n-1}{r_{\delta_m}}-r_{\delta_m}\bigl)x^{\prime}_{\delta_m}
+\lambda>\lambda.
\end{equation}
Integrating \eqref{eq:12-5-2} from $s=s_3$ to $s_1$, we get
\begin{equation}
\dfrac{\pi}{2}\geq\int_{s_3}^{s_1}\kappa ds
\geq \lambda\bigl(s_1-s_3\bigl)\geq\lambda\bigl(r_{\delta_m}(s_1)-r_{\delta_m}(s_3)\bigl).
\end{equation}
Hence,
\begin{equation}
r_{\delta_m}(s_1)\leq \frac{\pi}{2\lambda}+r_{\delta_m}(s_3).
\end{equation}
We have
\begin{equation}\label{eq:10-31-4}
r_{\delta_m}(s_3)\rightarrow +\infty
\end{equation}
because of  $r_{\delta_m}(s_1)\rightarrow +\infty$.

\noindent
Since $x_{\delta_m}^{\prime}(s_3)=0$ and $x_{\delta_m}(s_1)=0$ hold, from
$r_{\delta_m}(s_3)\rightarrow +\infty$, for some $\delta_m$, which is very near $\delta^{*}$,
we know that there exists an $s_5$ with $s_3<s_5<s_1$
such that
\begin{equation}
x_{\delta_m}^{\prime}(s_5)=-\sin(\dfrac{1}{\sqrt{r_{\delta_m}(s_3)}}).
\end{equation}
If we integrate \eqref{eq:12-5-2} from $s=s_3$ to $s_5$,  we  obtain

\begin{equation}\label{eq:11-3-10}
\dfrac{1}{\sqrt{r_{\delta_m}(s_3)}}=\int_{s_3}^{s_5}\kappa ds\geq \lambda(s_5-s_3).
\end{equation}
Since $r_{\delta_m}(s_3)\rightarrow +\infty$ holds,
\begin{equation}\label{eq:12-5-3}
\mid x^{\prime}(s_5)\mid=\sin \dfrac{1}{\sqrt{r_{\delta_m}(s_3)}}
>\dfrac{1}{2\sqrt{r_{\delta_m}(s_3)}}>\dfrac{1}{2\sqrt{r_{\delta_m}(s_5)}}
\end{equation}
yields

\begin{align}\label{eq:10-31-2}
\bigl(\dfrac{n-1}{r_{\delta_m}(s)}-r_{\delta_m}(s)\bigl)x^{\prime}_{\delta_m}(s)&\geq\bigl(\dfrac{n-1}{r_{\delta_m}(s_5)}-r_{\delta_m}(s_5)\bigl)x^{\prime}_{\delta_m}(s_5)\notag\\
             &>\dfrac{n-1}{r_{\delta_m}(s_5)}x^{\prime}_{\delta_m}(s_5)+\dfrac{1}{2}\sqrt{r_{\delta_m}(s_5)}\\
             &> \dfrac{1}{4}\sqrt{r_{\delta_m}(s_5)}.\notag
\end{align}
for $s_5\leq s\leq s_1$ since $r_{\delta_m}(s_5)\rightarrow +\infty$.

\noindent
From the equations \eqref{eq:9-4-3} and \eqref{eq:10-31-2}, we have
\begin{equation}\label{eq:10-31-3}
\kappa=-\dfrac{x^{\prime\prime}_{\delta_m}}{r^{\prime}_{\delta_m}}=x_{\delta_m}r^{\prime}_{\delta_m}+\bigl(\dfrac{n-1}{r_{\delta_m}}-r_{\delta_m}\bigl)x^{\prime}_{\delta_m}+\lambda>\dfrac{1}{4}\sqrt{r_{\delta_m}(s_5)}
\end{equation}
as $r_{\delta_m}\rightarrow \infty$.
Integrating  \eqref{eq:10-31-3} from $s=s_5$ to $s_1$, we have
\begin{equation}\label{eq:11-3-11}
\frac{\pi}{2}>\int_{s_5}^{s_1}\kappa ds>\frac{1}{4}\sqrt{r_{\delta_m}(s_5)}(s_1-s_5).
\end{equation}
Thus, one obtains from \eqref{eq:11-3-10} and  \eqref{eq:11-3-11}

\begin{align}\label{eq:10-31-5}
\max (x_{\delta_m})&=x_{\delta_m}(s_3)=x_{\delta_m}(s_3)-x_{\delta_m}(s_1)\notag\\
             &\leq s_1-s_5+s_5-s_3\notag\\
             &\leq \dfrac{2\pi}{\sqrt{r_{\delta_m}(s_5)}}+\dfrac{1}{\lambda\sqrt{r_{\delta_m}(s_3)}}\\
             &\leq \biggl(2\pi+\dfrac{1}{\lambda}\biggl)\dfrac{1}{\sqrt{r_{\delta_m}(s_3)}}.\notag
\end{align}
We conclude from \eqref{eq:10-31-4} and \eqref{eq:10-31-5} that $\max (x_{\delta_m})=x_{\delta_m}(s_3)\rightarrow 0$ if $r_{\delta_m}(s_1)\rightarrow +\infty$.
Hence,  $\gamma_{\delta_m}=(x_{\delta_m}, r_{\delta_m})$ gets close to the $r$-axis, its tangents also must converge to the $r$-axis. It is impossible since $\gamma_{\delta_m}=(x_{\delta_m}, r_{\delta_m})$ converges to $\gamma_{\delta^{*}}$ which is not $r$-axis.
This finishes our proof.
\end{proof}

\section {An upper bound of $x_\delta(s)$}
\begin{lemma}
\begin{equation}
\sup\limits_{\delta_0<\delta<\delta^{*}} \sup_{0<s<s_1}x_\delta(s)\leq\frac{\pi}{2\lambda}.
\end{equation}
\end{lemma}

\begin{proof} For $\delta>0$,
letting  $x^{\prime}(s_3)=0$ with  $0<s_3<s_1$,   we have
\begin{equation}\label{eq:10-28-1}
x^{\prime}(s)>0  \ \ \ \ \ \ {\mbox{ for} \ 0<s<s_3},
\end{equation}
 \begin{equation}\label{eq:10-28-2}
 x^{\prime}(s)<0 \ \ \ \ \ \ {\mbox{for} \ s_3<s<s_1}.
 \end{equation}
If  $r(s_3)\leq\sqrt{n-1}$,  we see from $r^{\prime}(s)>0$ that $r(s)<\sqrt{n-1}$ for $0<s<s_3$.
Thus, it follows from  \eqref{eq:9-4-3} and \eqref{eq:10-28-1} that
\begin{equation}\label{eq:12-5-4}
\kappa=-\dfrac{x^{\prime\prime}}{r^{\prime}}=xr^{\prime}+\bigl(\dfrac{n-1}{r}-r\bigl)x^{\prime}+\lambda>\lambda.
\end{equation}
for $0<s<s_3$.

\noindent By integrating \eqref{eq:12-5-4} from $s=0$ to $s_3$, we have
\begin{equation}
\frac{\pi}{2}=\int_{0}^{s_3}\kappa ds
\geq s_3\lambda \geq\lambda x(s_3)=\lambda \sup\limits_{0<s<s_1} x_\delta(s).
\end{equation}
\begin{equation}\label{eq:12-5-6}
\sup\limits_{0<s<s_1} x_\delta(s)=x(s_3)\leq\frac{\pi}{2\lambda}.
\end{equation}
If $r(s_3)>\sqrt{n-1}$ for some $\delta>0$,
we have from $r^{\prime}(s)>0$ that $r(s)>\sqrt{n-1}$ for $s_3<s<s_1$.
Then it follows from  \eqref{eq:9-4-3} and \eqref{eq:10-28-2} that
\begin{equation}\label{eq:12-5-5}
\kappa=-\dfrac{x^{\prime\prime}}{r^{\prime}}=xr^{\prime}+\bigl(\dfrac{n-1}{r}-r\bigl)x^{\prime}+\lambda>\lambda
\end{equation}
for $s_3<s<s_1$.

\noindent By integrating  \eqref{eq:12-5-5}  from $s=s_3$ to $s_1$, we have, from the mean value theorem,
\begin{equation}\label{eq:12-5-7}
\frac{\pi}{2}\geq\int_{s_3}^{s_1}\kappa ds \geq \lambda (s_1-s_3)
\geq\lambda x(s_3)
=\lambda \sup\limits_{0<s<s_1} x_\delta(s).
\end{equation}
From \eqref{eq:12-5-6} and \eqref{eq:12-5-7}, we get, for any $\delta>0$,

\begin{equation*}
\sup\limits_{0<s<s_1} x_\delta(s)=x(s_3)\leq\frac{\pi}{2\lambda},
\end{equation*}
 that is,

\begin{equation*}
\sup_{\delta_0<\delta<\delta^{*}}\sup\limits_{0<s<s_1} x_\delta(s)\leq\frac{\pi}{2\lambda}.
\end{equation*}
This completes the proof of the lemma.
\end{proof}

\section {Proof of the theorem}
\begin{proof}

\noindent  From the above lemmas, we have found that the profile curve $\gamma_\delta(s)=(x_\delta,r_\delta)$ ($0\leq s\leq s_1$) stay away from the $x$-axis, and remain bounded as $\delta\rightarrow \delta^{*}$. Therefore, the limiting profile curve $\gamma^{*}(s)=\gamma_{\delta^{*}}(s)$ ($0\leq s\leq s_1$) begins and ends on the $r$-axis, i.e., from $(0,\delta^{*})$ to $(0,r^{*})$, where $r^{*}=r_{\delta^{*}}(s_1)$.

\noindent
We now claim that the profile curve has the horizontal tangent, that is, $x_{\delta^{*}}^{\prime}(s_1)=-1$. From the definition of $\delta^{*}$, we obtain that $x_{\delta^{*}}^{\prime}(s_1)\leq 1$. If $x_{\delta^{*}}^{\prime}(s_1)<1$, one can choose $\delta$ such that $\delta>\delta^{*}$ and near $\delta^{*}$. Then one can still obtain a profile curve $\gamma_\delta$, which, in the first quadrant, is a graph $x=f_\delta(r)$, and which hits the $r$-axis in finite time. This contradicts the definition  of $\delta^{*}$. Hence, the profile curve $\gamma^{*}$ has the horizontal tangent, that is, $x_{\delta^{*}}^{\prime}(s_1)=-1$.
We can get that the profile curve $\gamma$
obtained by reflecting $\gamma_{\delta^{*}}([0,s_1])$ in the $r$-axis is a simple and closed curve in the upper half plane. This finishes
our proof of the theorem 1.1.

\end{proof}

%
%\begin{figure}[ht]
%\centerline{\includegraphics[width=6cm]{11.eps}\includegraphics[width=6cm]{22.eps}}
%\caption{The two images show the graph of $\gamma$ when $n=10,\lambda=0.5$ and $n=10, \lambda=-0.5$, respectively.}
%\label{gf}
%\end{figure}
%

\end {document}